\documentclass[11pt, reqno]{amsart}
\usepackage{amsthm, amsmath, amsfonts, amssymb, amsbsy, bigstrut, graphicx, enumerate, upref, longtable, comment, csquotes, bm}
\usepackage{float}
\usepackage{subcaption}
\usepackage{pstricks}
\usepackage[breaklinks]{hyperref}
\usepackage[numbers, sort&compress]{natbib}

\newtheorem{theorem}{Theorem}[section]
\newtheorem{lemma}[theorem]{Lemma}

\newcommand{\argmin}{\operatorname{argmin}}

\newcommand{\pkg}[1]{{\fontseries{b}\selectfont #1}} 

\newcommand{\bbx}{\mathbf{x}}
\newcommand{\bba}{\mathbf{A}}
\newcommand{\bbB}{\mathbf{B}}

\newcommand{\ee}{\mathbb{E}}

\newcommand{\pp}{\mathbb{P}}

\newcommand{\rr}{\mathbb{R}}

\newcommand{\MSE}{\text{MSE}}
\numberwithin{equation}{section}

\begin{document}
\title{Optimal choice of $k$ for $k$-nearest neighbor regression}
\author{Mona Azadkia}

\address{\newline Department of Statistics \newline Stanford University\newline Sequoia Hall, 390 Jane Stanford Way \newline Stanford, CA 94305\newline \newline \textup{\tt mazadkia@stanford.edu}}

\begin{abstract}
The $k$-nearest neighbor algorithm ($k$-NN) is a widely used non-parametric method for classification and regression. Finding an optimal $k$ in $k$-NN regression on a given dataset is a problem that has received considerable attention in the literature. Several practical algorithms for solving this problem have been suggested recently. The main result of this paper shows that the value of $k$ obtained by the simple and quick leave-one-out cross-validation (LOOCV) procedure is optimal under fairly general conditions. 
\end{abstract}
\maketitle

\section{Introduction}
Non-parametric regression is an important problem in statistics and machine learning \cite{wassermann06, tsybakov04, gkkw02}. The $k$-nearest neighbors algorithm ($k$-NN) is a popular non-parametric method of classification and regression. For a given set of $n$ data points $(\bbx_i, y_i) \in\rr^d\times\rr$, where $\bbx_i$'s are deterministic measurements and $y_i$'s are noisy observations, then for a point $\bbx\in\rr^d$ the $k$-NN algorithm outputs 
	\begin{eqnarray}
	\hat{y} = m_{k,n}(\bbx) = \frac{1}{k}\sum_{j\in N_k(\bbx)} y_j  
	\end{eqnarray}
as an estimate of $m(\bbx) := \ee[y\mid \bbx]$, where $N_k(\bbx)$ is the set of indices of the $k$ nearest neighbors of $\bbx$ among the $\bbx_i$'s and $\ee[y\mid \bbx]$ denotes the expected value of the response given that the vector of predictors equals $\bbx$. 

	The consistency of $k$-NN estimator for classification and regression has been studied by many researchers, some examples are ~\cite{beck79, BhattacharyaMack87, BickelBreiman83, c84, Collomb79, Collomb80, Collomb81, Cover68a, ch67, Devroye78a, Devroye81, D82, DevroyGyorfi85, DGKL94, Guerre00, Mack81, s77, s84, Zhao87, fh89}. The $k$-NN estimator with a fixed value of $k$ was first analyzed in \cite{ch67}. For $k = 1$, under mild assumptions the risk of the nearest neighbor estimator is as twice as the Bayes risk~\cite{Cover68a}. In general, as long as $k$ is fixed, the risk of the $k$-NN estimator does not converge to the Bayes risk~\cite{BD15}. This is an intuitive result, since for fixed value of $k$, the variance of the estimator does not converge to zero. It was proved in~\cite{s77} that if $k$ is chosen such that $k\rightarrow\infty$ and $k/n\rightarrow 0$ as $n\rightarrow\infty$, the $k$-NN estimator is universally consistent and the risk converges to the Bayes risk. The extra assumption $k/n\rightarrow 0$ is to control the bias of the estimator~\cite{s77}. Later it was shown that for universal consistency we only need $k/\log n\rightarrow 0$ as $n\to\infty$ \cite{DGKL94}.
		
	The problem of finding the optimal growth rate for $k$ and finiding the convergence rate of the error has been studied by many researchers, some examples are \cite{DGKL94, Gyorfi81, Mack83, Kulkarni95, gg02, cd14, bcg10, dgw17, Krzyzak86, Cheng95, BD15, GKM16}. As a result of these efforts, the theoretically optimal value of $k$ is now quite well-understood in various circumstances. But from a practical point of view, these results are hard to implement. This is because the theoretically optimal choice of $k$ often involves knowledge that is not available to the user. For example, it usually involves the error variance, but we cannot get our hands on it before solving the regression problem. 
	
	To circumvent such issues, various interesting ways of estimating the optimal $k$ from the data have been suggested in recent years~\cite{CBS19, kpotufe11, zl19, zl19unpublished}. In, ~\cite{Gyorfi81} the convergence rate of the local risk of the $k$-NN estimator for a given query point $\bbx$ has been studied and it was shown that for $\alpha > 0$ and $k = n^{1/(1 + (d/2\alpha))}$ the rate of convergence in probability is at least $n^{-1/(2 + (d/\alpha))}$. It has been shown~\cite{BD15, dgw17} that if the density of $\bbx$ is bounded from zero on its support, then the convergence rate of the local risk and the risk are the same. In the general case, it is interesting to find the optimal value of $k$ for given query point $\bbx$. In~\cite{kpotufe11}, value of $k$ has been chosen adaptively for each query point $\bbx$ such that  the $k$-NN estimator nearly achieve the minimax rate at $\bbx$. This suggested value of $k$ requires a tuning parameter that depends on the unknown intrinsic dimension of the vicinity of $\bbx$. In the classification setting, and when we have a sample of unlabeled data, it has been suggetsed in~\cite{CBS19} to use the unlabeled data for estimating the density of the $\bbx_i$'s and then use that information for choosing the value of $k$. Unfortunately, in many situations, we do not have this extra unlabeled sample. With a similar idea and for both classification and regression purposes, in~\cite{zl19, zl19unpublished}, value of $k$ has been chosen adaptively for each query point. Although this work does not require an estimate of the density, it requires three tuning parameters which are chosen by the user. 
	
	The main contribution of this paper is a non-asymptotic error bound which shows that a very old and simple method of choosing $k$ by a certain kind of cross-validation, known as leave-one-out cross-validation (LOOCV), provides an optimal value of $k$ quickly and efficiently. The consistency of this procedure has been known for a long time~\cite{Li84}, but the fact that this method is able to produce an optimal $k$ was not known before. The advantage of this simple method is that it does not require estimating the density or choosing any tuning parameter. 

\section{Main result}\label{main}
Let $\bbx$ be a deterministic $d$-dimensional vector and 
\[
y = m(\bbx)+\epsilon,
\]
where $m:\rr^d\to \rr$ is a measurable function and $\epsilon$ is a mean-zero random variable that is independent of $\bbx$. Let $\mu:= m(\bbx)$. Assume that there exists a finite constant $K$ such that $\ee(\exp({\epsilon^2/K})) \leq 2$. 

Let $\epsilon_1, \cdots, \epsilon_n$ be a sampe of i.i.d.~copies of $\epsilon$. For $n$ deterministic measurements $\bbx_1, \cdots, \bbx_n$, we have a data set of $n$ pairs $(\bbx_1, y_1), \cdots, (\bbx_n, y_n)$. For each $i$, let $N_k(i)$ be the indices of the $k$ nearest neighbors of $\bbx_i$ among $\bbx_1,\ldots,\bbx_{i-1},\bbx_{i+1},\ldots, \bbx_n$, where ties are broken at random. Define
\begin{eqnarray*}
\hat{m}_{k,n}(\bbx_i) := \frac{1}{k}\sum_{j\in N_k(i)} y_j.
\end{eqnarray*}
We define the mean squared error of this estimator as
\begin{eqnarray*}
\text{MSE}(k) = \ee[\frac{1}{n}\sum_{i=1}^n(m(\bbx_i) - \hat{m}_{k,n}(\bbx_i))^2],
\end{eqnarray*}
where the expectation is taken over the $\epsilon_i$'s. Note that this is not exactly the common mean squared error defined in the literature, since we are excluding $\bbx_i$ from its set of nearest neighbors. We are considering this to be the $\text{MSE}$, since the $\bbx_i$'s are deterministic, and so there is no concept of a new $\bbx$ from the distribution of the $\bbx_i$'s, and using $\bbx_i$ in the estimate for $\mu_i$ would cause overfitting. 

Our object of interest is the number 
\begin{eqnarray}\label{kstar}
k^* := \argmin_{1\le k\le n - 1} \text{MSE}(k).
\end{eqnarray}
Of course, we cannot directly compute $k^*$ from the data since the function $m$ is unknown. Instead, we produce a surrogate. Define 
\begin{eqnarray*}
	f(k) := \frac{1}{n}\sum_{i=1}^n\biggl(y_i - \frac{1}{k}\sum_{j\in N_k(i)}y_j\biggr)^2, 
\end{eqnarray*}
and let 
\begin{eqnarray}\label{ktilde}
\tilde{k} := \argmin_{1\le k\le n-1} f(k).
\end{eqnarray}
Note that $\tilde{k}$, unlike $k^*$,  is computed from the data. The intention is to use $\tilde{k}$ as the chosen value of $k$ in $k$-NN regression. This procedure for selecting $k$ is known as leave-one-out cross-validation (LOOCV). The following theorem, which is the main result of this paper, shows that 
\[
|\MSE(k^*)- \MSE(\tilde{k})| = O\biggl(\sqrt{\frac{\log n}{n}}\biggr) 
\]
when $K$ and $d$ are fixed. One of the main strengths of this theorem is that no other condition is needed.  
\begin{theorem}\label{probBound}
Let $K$, $k^*$, and $\tilde{k}$ be as above and $\bm\mu = (\mu_1, \cdots, \mu_n)$ where $\mu_i = m(\bbx_i)$. Then there are positive constants $A$, $B$ and $C$ depending on $d$ and $K$, such that for any $t\ge 0$, 
	\begin{eqnarray*}
	\lefteqn{\pp(|\emph{\MSE}(k^*) - \emph{\MSE}(\tilde{k})| \geq t) \le }\\
	&& 4n\exp({-n\min\{At^2, Bt\}}) + 4n\exp({-Cn^2t^2 / \|\bm\mu\|^2}).
	\end{eqnarray*}
\end{theorem}
A remarkable consequence of the above theorem is that the choice of $k$ by LOOCV adapts automatically to the smoothness of the regression function $m$, because the bound on the right does not depend on the smoothness of $m$.  

In many situations, $\MSE(k^*)$ is much greater than $n^{-1/2}(\log n)^{1/2}$. For example, by \cite[Theorem 3.2]{gkkw02}, for Lipschitz functions with bounded support, the lower minimax rate of convergence (in terms of $\MSE$) is $O(n^{-2/(2+d)})$ and for $d\geq 3$. In such cases, this result implies that $\MSE(k^*)/\MSE(\tilde{k}) \to 1$ as $n\to\infty$.

\section{Proof}\label{proofs}
For a matrix $\bba$, recall that the $2$-norm $\|\bba\|_2$ and the Frobenius norm $ \|\bba\|_{\text{F}}$ are defined as
\begin{eqnarray*}
	\|\bba\|_2 = \sup_{\|\bbx\| = 1} \|\bba\bbx\|, \qquad \|\bba\|_{\text{F}} = \biggl(\sum_{i, j}a_{ij}^2\biggr)^{1/2}.
\end{eqnarray*}
Throughout this proof, $\gamma_d$ will denote any constant that depends only on $d$. The value of $\gamma_d$ may change from line to line.

	Let $\mu_i := m(\bbx_i)$, and $\bm\epsilon = (\epsilon_1, \ldots, \epsilon_n)$ and $\bm\mu = (\mu_1, \ldots, \mu_n)$.  By writing $y_i = \mu_i + \epsilon_i$, we have	
	\begin{eqnarray*}
	f(k) &=& \frac{1}{n}\sum_{i=1}^n\biggl(y_i - \frac{1}{k}\sum_{j\in N_k(i)}y_j\biggr)^2 \\
	&=& \frac{1}{n}\sum_{i=1}^n\biggl(\mu_i + \epsilon_i - \frac{1}{k}\sum_{j\in N_k(i)}y_j\biggr)^2 \\
	&=& \frac{1}{n}\sum_{i=1}^n\biggl[\biggl(\mu_i - \frac{1}{k}\sum_{j\in N_k(i)}y_j\biggr)^2 + \epsilon_i^2 + 2\epsilon_i\biggl(\mu_i - \frac{1}{k}\sum_{j\in N_k(i)}y_j\biggr)\biggr]. 
	\end{eqnarray*}
	Since the $\epsilon_i$'s are independent with mean zero, taking expectation on both sides gives
	\begin{eqnarray}\label{fdecomposition}
	\ee[f(k)] = \ee[\epsilon^2] + \MSE(k).
	\end{eqnarray}
	Define $g(k) := \ee[f(k)]$. By definition of $k^*$,  $\MSE(k^*)\leq\MSE(k)$ for all $k$. In particular, $\MSE(k^*)\leq\MSE(\tilde{k})$, which implies that $g(k^*)\leq g(\tilde{k})$. Also by definition of $\tilde{k}$, we have $f(\tilde{k}) \leq f(k^*)$. Putting these two together, we get
	\begin{eqnarray}
	\lefteqn{\pp(|\MSE(k^*) - \MSE(\tilde{k})| \geq t)} \nonumber\\ 
	& & = \pp(g(\tilde{k}) - g(k^*)  \geq t)\nonumber\\
	& & \leq \pp(g(\tilde{k}) - g(k^*)  \geq t + f(\tilde{k}) - f(k^*)) \nonumber \\
	& & \leq \pp(|g(k^*) - f(k^*)| \geq t/2) + \pp(|g(\tilde{k}) - f(\tilde{k})| \geq t/2) \nonumber\\
	& & \leq 2\sum_{k=1}^{n-1}\pp(|f(k)-g(k)|\ge t/2). \nonumber
	\end{eqnarray} 
Thus, the proof will be complete if we can prove the following lemma. 	
	\begin{lemma}\label{concentraionBoundExpected}
	There are positive constants $A$, $B$ and $C$ depending on $d$ and $K$ such that for any $1\le k\le n-1$ and any $t > 0$, 
	\begin{eqnarray*}
	\pp(|f(k) - g(k)|\geq t) &\leq & 2e^{-n\min\{At^2, Bt\}} + 2e^{-Cn^2t^2 / \|\bm\mu\|^2}.
	\end{eqnarray*}
	\end{lemma}

	Define a nonsymmetric $n\times n$ matrix $\bbB = [b_{ij}]$ as	\begin{eqnarray}\label{Bdef}
	b_{ij} := \left\{ 
	\begin{array}{ll}
	1 & \qquad i = j, \\
	0 & \qquad j\not\in N_k(i), \\
	-1/k & \qquad j\in N_k(i).
	\end{array}\right. 
	\end{eqnarray}
	Let $\bba = [a_{ij}] = \bbB^T\bbB/n$. Then we can rewrite $f(k)$ in the following form:
	\begin{eqnarray}
	f(k) = (\bm\epsilon + \bm\mu)^T \bba(\bm\epsilon + \bm\mu).
	\end{eqnarray}
	Using the triangle inequality, we have
	\begin{eqnarray*}
	|f(k) - g(k)| &=& |\bm\epsilon^T \bba\bm\epsilon - \ee[\bm\epsilon^T \bba\bm\epsilon] + 2\bm\epsilon^T \bba\bm\mu| \\
	&\leq & |\bm\epsilon^T \bba\bm\epsilon - \ee[\bm\epsilon^T \bba\bm\epsilon]| + 2|\bm\epsilon^T \bba\bm\mu|.
	\end{eqnarray*}
Therefore it is enough to find probability tail bounds on $|\bm\epsilon^T \bba\bm\epsilon - \ee[\bm\epsilon^T \bba\bm\epsilon]|$ and $|\bm\epsilon^T \bba\bm\mu|$. To find such bounds we need to have bounds on the Frobenius norm and the  $2$-norm of $\bba$. The following lemmas give such bounds.
	\begin{lemma}\label{normFrobeniousA}
	For the matrix $\bba$ defined above, 
	\begin{eqnarray}
	\|\bba\|_{\text{F}}^2 \leq \frac{\gamma_d}{n},
	\end{eqnarray}
	where $\gamma_d$ is a constant that only depends on $d$. 
	\end{lemma}
	
	\begin{lemma}\label{Norm2A}
	For the matrix $\bba$ defined above, 
	\begin{eqnarray}
	\|\bba\|_2 \leq \frac{\gamma_d}{n},
	\end{eqnarray}
	where $\gamma_d$ is a constant that only depends on $d$.
	\end{lemma}

We will prove these lemmas later. 
	\begin{proof}[Proof of Lemma \ref{concentraionBoundExpected}]

Throughout this proof as usual, $\gamma_d$ denotes any constant that depends only on $d$, and $c$ will denote any universal constant.  First, let us obtain a tail bound for $|\bm\epsilon^T \bba\bm\epsilon - \ee[\bm\epsilon^T \bba\bm\epsilon]|$.
	By the Hanson--Wright inequality \cite{rv13} and Lemmas~\ref{normFrobeniousA} and \ref{Norm2A}, we have
	\begin{eqnarray}\label{HansonWright}
  	\lefteqn{\pp(|\bm\epsilon^T \bba\bm\epsilon - \ee[\bm\epsilon^T \bba\bm\epsilon]| \geq t)}\nonumber\\
  	& & \leq 2\exp\left(-c\min\left\{\frac{t^2}{K^2\|\bba\|_{\text{F}}^2}, \frac{t}{K\|\bba\|_2}\right\}\right) \nonumber\\
  	& & \leq 2\exp\left(-n\min\left\{\frac{t^2}{K^2\gamma_d}, \frac{t}{K_1\gamma_d}\right\}\right).
	\end{eqnarray} 
	An easy computation gives  
	\begin{eqnarray}\label{asequal}
		\ee[\bm\epsilon^T \bba\bm\epsilon] &=& \frac{1}{n}\sum_{i=1}^n \ee\biggl(\epsilon_i - \frac{1}{k}\sum_{j\in N_k(i)} \epsilon_j\biggr)^2\nonumber\\
		&=& \biggl(1+\frac{1}{k}\biggr) \ee(\epsilon^2).
	\end{eqnarray}	
	Putting \eqref{asequal} and \eqref{HansonWright} together gives us
	\begin{eqnarray}\label{firstPart}
	 \pp(|\bm\epsilon^T\bba\bm\epsilon - \ee[\bm\epsilon^T \bba\bm\epsilon]| \geq t) &=& \pp(|\bm\epsilon^T \bba\bm\epsilon - \ee[\bm\epsilon^T \bba\bm\epsilon]| \geq t) \nonumber\\
	 &\leq& 2\exp\left(-n\min\left\{\frac{t^2}{K^2\gamma_d}, \frac{t}{K\gamma_d}\right\}\right).
	\end{eqnarray}
Next, we obtain a tail bound for $|\bm\epsilon^T \bba\bm\mu|$. Remember that $\ee[e^{\epsilon_i^2/K}] \leq 2$ and therefore $\epsilon_i$'s are sub-Gaussian. Then by the equivalent properties of sub-Gaussian random variables \cite{v12}, there exist a constant $C$ that only depends on $K$ such that $\ee[e^{\lambda\epsilon}] \leq e^{\lambda^2C/2}$ for all $\lambda$. Then by using the Hoeffding bound for sub-Gaussian variables \cite[Proposition 2.5]{w19}, we have 
	\begin{eqnarray}\label{HoeffdingBase}
	\pp(|\bm\epsilon^T \bba\bm\mu| \geq t) & = & \pp\biggl(\biggl|\sum_{j=1}^n\biggl(\sum_{i=1}^n a_{ji}\mu_{i}\biggr)\epsilon_j \biggr| \geq t \biggr) \nonumber\\
	& \leq & 2\exp\left(-\frac{t^2}{2 C\sum_{j=1}^n(\sum_{i=1}^n a_{ji}\mu_{i})^2}\right) .
	\end{eqnarray}
	Note that 
	\begin{eqnarray}\label{denomHoeffding}
	\sum_{j=1}^n\biggl(\sum_{i=1}^n a_{ji}\mu_{j}\biggr)^2 & \leq & \|\bba\|_2^2 \|\bm\mu\|^2.
	\end{eqnarray}
	Inequalities \eqref{HoeffdingBase} and \eqref{denomHoeffding} together give
	\begin{eqnarray}
	\pp(|\bm\epsilon^T \bba\bm\mu| \geq t) & \leq & 2\exp\left(-\frac{C t^2}{2\|\bba\|_2^2 \|\bm\mu\|^2}\right).
	\end{eqnarray}
	Therefore by Lemma~\ref{Norm2A},
	\begin{eqnarray}\label{secondPart}
	\pp(|\bm\epsilon^T \bba\bm\mu | \geq t ) &\leq & 2\exp\left(-\frac{C n^2t^2}{\gamma_d\|\bm\mu\|^2}\right).
	\end{eqnarray}

	Combining \eqref{firstPart} and \eqref{secondPart}, we get
	\begin{eqnarray*}
	\pp(|f(k) - g(k)|\geq t) & \leq & 2\exp\left(-n\min\left\{\frac{t^2}{K^2\gamma_d}, \frac{t}{K \gamma_d}\right\}\right) + \\
	&& 2\exp\left(-\frac{C n^2t^2}{\gamma_d\|\bm\mu\|^2}\right).
	\end{eqnarray*}
	\end{proof}

\begin{proof}[Proof of Lemma \ref{normFrobeniousA}] 
	Let $\mathbf{b}_i$ be the $i$-th row of matrix $\bbB$. Then
	\begin{eqnarray*}
	\|\bba\|_{\text{F}}^2 &=&  \frac{1}{n^2}\sum_{i, j} \langle \mathbf{b}_i,\mathbf{b}_j\rangle^2 \\
		&=& \frac{1}{n}\biggl(1 + \frac{1}{k}\biggr) + \frac{1}{n^2}\sum_{i = 1}^n\sum_{j\neq i} \langle \mathbf{b}_i,\mathbf{b}_j\rangle^2.
	\end{eqnarray*}
	For any distinct $i, j$,  
	\begin{eqnarray*}
	\langle \mathbf{b}_i,\mathbf{b}_j\rangle = -\frac{1}{k}[1_{\{i\in N_k(j)\}} + 1_{\{j\in N_k(i)\}}] + \frac{1}{k^2}|N_k(i)\cap N_k(j)|,
	\end{eqnarray*}
	and therefore 
	\begin{eqnarray}
	|\langle \mathbf{b}_i,\mathbf{b}_j\rangle| & \leq & \frac{2}{k}.
	\end{eqnarray}
	This shows that for any $i$, 
	\begin{eqnarray}
	\sum_{j\neq i} \langle \mathbf{b}_i,\mathbf{b}_j\rangle^2 \leq \frac{4}{k^2}|\{j:  \langle \mathbf{b}_i,\mathbf{b}_j\rangle \neq 0\}|.
	\end{eqnarray}
	But if $\langle \mathbf{b}_i,\mathbf{b}_j\rangle \neq 0$, then 
	\begin{eqnarray}
	\left(\{i\}\cup N_k(i)\right) \cap \left(\{j\}\cup N_k(j)\right) \neq \emptyset.
	\end{eqnarray}
	By definition $|N_k(i)| = k$, and for any $\ell$, by  \cite[Corollary 6.1]{gkkw02} there are at most $\gamma_d k$ indices $j$ such that $\ell\in N_k(j)$. Therefore for any $i$, 
	\begin{eqnarray}
	|\{j:\langle \mathbf{b}_i, \mathbf{b}_j\rangle \neq 0\}| \leq \gamma_d k(k+1).
	\end{eqnarray}
	This gives the required bound on $\|\bba\|_{\text{F}}^2$.
	\end{proof}

	\begin{proof}[Proof of Lemma \ref{Norm2A}]
Take  any $\bbx$ such that $\|\bbx\| = 1$. Then by \cite[Corollary 6.1]{gkkw02}, 
	\begin{eqnarray*}
	\|\bbB \bbx\|^2 &=& \sum_{i=1}^n \langle \mathbf{b}_i, \bbx \rangle^2 \\
	&\leq &  2\sum_{i=1}^n x_i^2 + 2\sum_{i=1}^n\biggl(\frac{1}{k}\sum_{j\in N_k(i)}x_j\biggr)^2 \\
	&\leq & 2\|x\|^2 + \frac{2}{k}\sum_{i=1}^n\sum_{j\in N_k(i)}x_j^2 \\
	&=& 2\|x\|^2 + \frac{2}{k}\sum_{j=1}^n\sum_{i:j\in N_k(i)}x_j^2 \\
	&\leq & \gamma_d\|x\|^2.
	\end{eqnarray*}
	Therefore $\|B\|_2\le \gamma_d$ and hence $\|A\|_2\le \gamma_d/n$.
	\end{proof}
An R language package \pkg{knnopt} will soon be made available on the CRAN repository.

\section*{Acknowledgement}
	The author is very grateful to her advisor Sourav Chatterjee for his constant encouragement and insightful conversations and comments.

\end{document}